\documentclass[11pt]{article}
\usepackage[margin=1in]{geometry}
\usepackage{amsmath,amssymb,amsthm}
\usepackage{tikz}
\newcommand{\avec}{\operatorname{avec}}
\newcommand{\diam}{\operatorname{diam}}
\theoremstyle{plain}
\newtheorem{proposition}{Proposition}[section]
\newtheorem{lemma}{Lemma}[section]
\newtheorem{theorem}{Theorem}[section]
\newtheorem{conjecture}{Conjecture}[section]
\theoremstyle{remark}
\newtheorem{remark}{Remark}[section]

\begin{document}

\title{Average Eccentricity in $k$-Connected Tournaments}
\author{F.J. Osaye\\ Department of Mathematics and Statistics\\ Troy University, AL}
\date{}
\maketitle

\begin{abstract}
The average eccentricity $\avec(D)$ of a strong digraph $D$ is the mean eccentricity of all vertices in $D$. A tournament $T$ is an orientation of a complete graph. In this paper, we prove that the average eccentricity satisfies $2 \le \avec(T) < \frac{n}{2k} + \frac{11}{2} + \frac{k}{8n}$ for a $k$-connected tournament $T$ of order $n$. We further examine the case for strong tournaments where $k = 1$, and show that $\avec(T) \le \frac{n+11}{2} + \frac{1}{8n}$ for any strong tournament $T$. Moreover, we show that the lower bound $2$ is sharp and that the upper bounds are best possible apart from an additive constant.
\end{abstract}

\section{Introduction}

Let $D = (V,A)$ be a connected digraph (directed graph) of order $n$. $D$ is said to be strong if it is possible to reach any vertex from any other vertex in $D$. For a strongly connected digraph $D$ and $u,v \in V(D)$, $d_D(v,u)$ denotes the directed distance from $v$ to $u$. With $n = |V(D)|$, the average distance $\mu(D)$ is the mean distance between all ordered pairs of distinct vertices of $D$, i.e.,
\[
\mu(D) = \frac{1}{n(n-1)} \sum_{a \ne b} d_D(a,b).
\]
The eccentricity of vertex $v$ is
\[
\operatorname{ecc}_D(v) = \max_{u \in V(D)} d_D(v,u).
\]
The average eccentricity is
\[
\avec(D) = \frac{1}{n} \sum_{v \in V(D)} \operatorname{ecc}_D(v).
\]

The average eccentricity $\avec(G)$ for undirected graphs $G$ was first introduced by Buckley and Harary [1] under the name eccentric mean, while its systematic study was first carried out by Dankelmann et al.\ [4] who showed, among other results, that the path maximizes the average eccentricity among all connected undirected graphs of order $n$. Since then, several upper and lower bounds have been obtained on $\avec(G)$ in terms of other graph parameters. See, for example, ([2], [3], [5], [9], and [7]). However, not much is known about the average eccentricity for digraphs while much work has been done on the average distance for both directed and undirected graphs. For a survey of results on the average distance in undirected and directed graphs we refer to Plesn\'ik [10]. One of the results obtained by Plesn\'ik [10] showed that for a strong tournament $T$,
\[
\frac{3}{2} \le \mu(T) \le \frac{n+4}{6} + \frac{1}{n},
\]
where $\mu(T) = \frac{3}{2}$ if and only if $\diam(T) = 2$ and $\mu(T) = \frac{n+4}{6} + \frac{1}{n}$ when $\diam(T) = n-1$. The bound was later improved by Moon [8] by taking into account the degree of a vertex whose removal keeps the tournament strong. Plesn\'ik's [10] bound was further improved for higher connectivity by Dankelmann and Volkmann [6] who proved that a $k$-connected tournament $T$ of order $n$ has average distance
\[
\mu(T) < \frac{n}{6k} + \frac{19}{6} + \frac{k}{n},
\]
which is best possible apart from the additive constant. Their method groups vertices into distance layers from a diametral endpoint, bounds pairwise distances via the layer indices, and reduces to a vertex-weighted path. In this paper, we prove the analogous bounds on average eccentricity for both strong tournaments and tournaments of higher connectivity.

\section{Preliminaries}

All digraphs are finite. For a digraph $D = (V,A)$, $V(D)$ denotes the vertex set and $n = |V(D)|$ the order while $A(D)$ denotes the arc set. An arc $(u,v)$ is a directed edge with tail $u$ and head $v$; we write $u \to v$ if $(u,v)$ is an arc of $D$. Then $v$ is an out-neighbor of $u$ (and $u$ an in-neighbor of $v$). A tournament is an orientation of a complete graph where for every pair of distinct vertices $u,v$, exactly one of $u \to v$, $v \to u$ holds. This means that for every pair of distinct vertices, there is exactly one directed edge between them. Hence, a tournament with $n$ vertices has exactly $|A| = \binom{n}{2}$ arcs. A tournament is transitive if it contains no directed cycle; equivalently, its vertices can be labelled $v_1,\dots,v_t$ so that $v_i \to v_j$ if and only if $i < j$. This means a tournament is said to be transitive if $u \to v$ and $v \to w$ imply $u \to w$ for $u,v,w \in V$. The directed distance $d_D(u,v)$ is the number of arcs in a shortest directed $u \to v$ path, or $\infty$ if no such path exists. For $k \ge 1$, $D$ is $k$-connected if it remains strong after deleting any set of fewer than $k$ vertices. In other words, a digraph is said to be strong if for every pair of distinct vertices $u$ and $v$, there is a directed path from $u$ to $v$ and a directed path from $v$ to $u$. A set $S \subset V(D)$ is an $X$--$Y$ separator, for $X,Y \subset V(D)$, if every directed path starting in $X$ and ending in $Y$ contains a vertex of $S$. Every $k$-connected tournament, for $k \ge 1$, is strong; in particular 1-connectivity is exactly strong connectivity (deleting fewer than one vertex being vacuous). The converse fails: a strong tournament need not be 2-connected (e.g.\ deleting one vertex of $T_{n,1}$ disconnects it). Thus
\[
\text{strong} = \text{1-connected} \supsetneq \text{2-connected} \supsetneq \cdots \supsetneq \text{$k$-connected} \supsetneq \cdots.
\]
Let $D$ be strong and $v \in V(D)$. The eccentricity of $v$ is
\[
\operatorname{ecc}_D(v) = \max_{u \in V(D)} d_D(v,u),
\]
the greatest distance from $v$ to any vertex. We define:
\begin{itemize}
\item the diameter $\diam(D) = \max_v \operatorname{ecc}_D(v)$;
\item the radius $\operatorname{rad}(D) = \min_v \operatorname{ecc}_D(v)$.
\end{itemize}
We write $d_T$, $\operatorname{ecc}_T$, etc.\ when the digraph is a tournament $T$, dropping subscripts when the digraph is clear from context. The vertex set of the extremal constructions in this paper is partitioned into ordered classes called layers. Relative to any layering, and in particular to the distance layers below, movement from layer $i$ to layer $j$ is forward if $j > i$ and backward if $j < i$. A backward jump from a vertex $x$ in layer $i$ is an arc $x \to y$ with $y$ in some layer $j < i$. Arcs from a layer $L_i$ to a layer $L_j$ are complete bipartite if $x \to y$ for every $x \in L_i$ and every $y \in L_j$.

Fix $u \in V(T)$ with $\operatorname{ecc}_T(u) = D := \diam(T)$. The distance layers from $u$ are
\[
V_i = \{w \in V(T) : d_T(u,w) = i\}, \quad n_i = |V_i|, \quad i = 0,\dots,D,
\]
which partition $V(T)$; thus $V_0 = \{u\}$, $n_D \ge 1$, and $\sum_{i=0}^D n_i = n$.

\section{Main Results}

We begin by presenting preliminary results useful for proving the main theorems.

\begin{proposition}
Let $D = \diam(T)$ be the diameter of the $k$-connected tournament $T$ of order $n$. Then:
\[
D \le \frac{n+k-2}{k}.
\]
\end{proposition}

\begin{proof}
Let $D = \diam(T)$ and fix $u \in V(T)$ with $\operatorname{ecc}_T(u) = D$. Choose $v$ with $d_T(u,v) = D$. For $i = 0,\dots,D$ define the distance layers
\[
V_i = \{w \in V(T) : d_T(u,w) = i\}, \quad n_i = |V_i|.
\]
Then $V_0 = \{u\}$, $v \in V_D$ so $n_D \ge 1$, $\sum_{i=0}^D n_i = n$, and the $V_i$ partition $V(T)$. Since $T$ is $k$-connected, every $V_i$ with $1 \le i \le D-1$ is a $u$--$V_D$ separator; by Menger's theorem,
\begin{equation}
n_i \ge k \qquad (1 \le i \le D-1).
\end{equation}
Consequently,
\begin{equation}
n = \sum_{i=0}^D n_i \ge 1 + k(D-1) + 1,
\end{equation}
Thus,
\begin{equation}
D \le \frac{n+k-2}{k}.
\end{equation}
\end{proof}

The bound (3) in Proposition 3.1 is best possible. Equality requires $n_0 = n_D = 1$ and $n_i = k$ for $1 \le i \le D-1$, i.e.\ $n = 2 + k(D-1)$. For $D \ge 4$, this is attained by the following $k$-connected tournament: layers $L_0 = \{u\}$, $L_1,\dots,L_{D-1}$ with $|L_i| = k$, $L_D = \{v\}$, with complete bipartite arcs $L_i \to L_{i+1}$, complete bipartite backward arcs $L_i \to L_j$ for $j \le i-2$, and a transitive tournament within each layer. As in Proposition 3.3, exactly one arc joins each pair. It is $k$-connected: deleting fewer than $k$ vertices leaves every interior layer nonempty; if both endpoints are deleted, the remaining interior layers (at least three, since $D \ge 4$) are still strongly connected via the forward chain and the backward skip from $L_{D-1}$ to $L_1$, and if exactly one endpoint is deleted, the return path is supplied by the backward skip from the last surviving layer to the first; and $d(u,v) = D$ is the diameter, since forward motion advances one layer per step and every other pair is joined by a path of length at most $D$ (a backward jump costs $1$, a return to the previous layer costs $2$, and a return within a layer costs at most $3 \le D$). For $k = 1$ this is $T_{n,1}$ itself ($D = n-1$); for $k = 2$ and $D = 3$, equality is attained by $T_{6,2}$ ($D = 3$); for $k \ge 3$, $T_{n,k}$ cannot attain equality since $n = mk$ makes $\frac{n+k-2}{k}$ non-integral.

We shall use the following estimate, proved by Dankelmann and Volkmann in [6]:

\begin{lemma}[Dankelmann--Volkmann]
For $x \in V_i$ and $y \in V_j$,
\begin{equation}
d_T(x,y) + d_T(y,x) \le |j-i| + 5.
\end{equation}
In particular, $d_T(x,y) \le |i-j| + 4$ for $x \ne y$.
\end{lemma}

The next lemma is specific to tournaments and has no analogue for general digraphs. It is the key to the constant $1/(2k)$.

\begin{lemma}
Let $x \in V_i$ and $y \in V_j$ with $j \le i-2$. Then $x \to y$ is an arc of $T$; in particular $d_T(x,y) = 1$.
\end{lemma}

\begin{proof}
Suppose instead that $y \to x$. Concatenating a shortest directed $u \to y$ path (of length $d_T(u,y) = j$) with the arc $y \to x$ yields a directed $u \to x$ walk of length $j+1$, so $d_T(u,x) \le j+1$. But $d_T(u,x) = i$ and $j \le i-2$ imply $j+1 \le i-1 < i$, a contradiction. Since $T$ is a tournament, exactly one of $x \to y$, $y \to x$ holds; hence $x \to y$.
\end{proof}

\begin{remark}
Lemma 3.2 asserts that from layer $V_i$, every vertex in layers $V_0,\dots,V_{i-2}$ is an out-neighbour. Backward jumps of at least two layers are always direct arcs.
\end{remark}

\subsection{Bounding eccentricity by forward distance}

Let $x \in V_i$. We bound $d_T(x,y)$ according to the layer of $y$.

\begin{proposition}
For $x \in V_i$,
\begin{equation}
\operatorname{ecc}_T(x) \le (D-i) + 5.
\end{equation}
\end{proposition}

\begin{proof}
Let $y \in V_j$.
\begin{itemize}
\item If $j \le i-2$: $d_T(x,y) = 1$ by Lemma 3.2.
\item If $j \in \{i-1,i\}$: $d_T(x,y) \le |i-j| + 4 \le 5$ by Lemma 3.1.
\item If $j > i$: $d_T(x,y) \le (j-i) + 4 \le (D-i) + 4$ by Lemma 3.1, since $j \le D$.
\end{itemize}
In each case $d_T(x,y) \le (D-i) + 5$ (using $D-i \ge 0$). Taking the maximum over $y$ gives (5).
\end{proof}

\begin{remark}
Proposition 3.2 bounds eccentricity by the forward distance $D-i$ plus a constant. The backward direction contributes only $O(1)$, by Lemma 3.2. This is strictly stronger than the symmetric path estimate $\max(i, D-i) + O(1)$.
\end{remark}

Summing (5) over all vertices yields:
\begin{equation}
n \cdot \avec(T) = \sum_x \operatorname{ecc}_T(x) \le \sum_{i=0}^D n_i(D-i) + 5n.
\end{equation}

\begin{lemma}
\[
\sum_{i=0}^D n_i(D-i) \le \frac{n^2}{2k} + \frac{n}{2} + \frac{k}{8}.
\]
\end{lemma}

\begin{proof}
Since $\sum_{i=0}^D n_i(D-i) = Dn - \sum_{i=0}^D i n_i$ and $n_0 \cdot 0 = 0$, (1) together with $n_D \ge 1$ gives
\[
\sum_{i=0}^D i n_i \ge \sum_{i=1}^{D-1} i n_i + D n_D \ge k \sum_{i=1}^{D-1} i + D = \frac{kD(D-1)}{2} + D.
\]
Hence
\[
\sum_{i=0}^D n_i(D-i) \le Dn - \frac{kD(D-1)}{2} - D = g(D) - D < g(D),
\]
where $g(D) := Dn - \frac{k}{2}D^2 + \frac{k}{2}D$. The function $g(D) = -\frac{k}{2}D^2 + \left(n + \frac{k}{2}\right)D$ is quadratic in $D$ whose leading coefficient is negative, with first derivative
\[
g'(D) = -kD + \left(n + \frac{k}{2}\right),
\]
and second derivative $g''(D) = -k$. So,
\[
g(D) = -\frac{k}{2}(D - D^*)^2 + g(D^*),
\]
with equality if and only if $D = D^*$. Setting the first derivative to zero shows that $g(D)$ is maximum at $D^* = \frac{n}{k} + \frac{1}{2}$, of value
\[
g(D^*) = -\frac{k}{2}[D^*]^2 + \left(n + \frac{k}{2}\right)D^* = \frac{n^2}{2k} + \frac{n}{2} + \frac{k}{8}.
\]
Since $g(D) - D < g(D) \le g(D^*)$, the lemma follows, and in fact $\sum_{i=0}^D n_i(D-i) < g(D^*)$ always holds.
\end{proof}

\subsection{Bounds on $k$-connected tournaments}

\begin{theorem}
Let $T$ be a $k$-connected tournament of order $n \ge 3$. Then
\begin{equation}
\avec(T) < \frac{n}{2k} + \frac{11}{2} + \frac{k}{8n}.
\end{equation}
This bound is best possible apart from the additive constant.
\end{theorem}

\begin{proof}
From (6) and Lemma 3.3,
\[
n \cdot \avec(T) \le \frac{n^2}{2k} + \frac{n}{2} + \frac{k}{8} + 5n,
\]
i.e.\
\[
\avec(T) \le \frac{n}{2k} + \frac{11}{2} + \frac{k}{8n}.
\]
The inequality is strict: the proof of Lemma 3.3 gives $\sum_{i=0}^D n_i(D-i) \le g(D) - D < g(D^*)$, so from (6),
\[
n \cdot \avec(T) < \frac{n^2}{2k} + \frac{n}{2} + \frac{k}{8} + 5n,
\]
hence
\[
\avec(T) < \frac{n}{2k} + \frac{11}{2} + \frac{k}{8n},
\]
proving (7).
\end{proof}

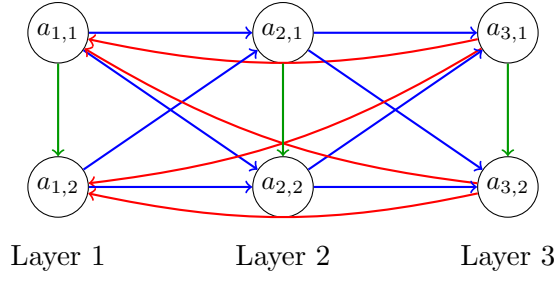
\begin{figure}[htbp]
\centering
\begin{tikzpicture}[scale=0.85, v/.style={circle, draw, minimum size=0.8cm, inner sep=1pt, font=\small}]
\foreach \i/\x in {1/0, 2/3.5, 3/7} {
  \foreach \j/\y in {1/1.2, 2/-1.2} {
    \node[v] (a\i\j) at (\x,\y) {$a_{\i,\j}$};
  }
}
\foreach \j in {1,2} {
  \foreach \k in {1,2} {
    \draw[->, blue, thick] (a1\j) -- (a2\k);
    \draw[->, blue, thick] (a2\j) -- (a3\k);
    \draw[->, red, thick] (a3\j) to[bend left=12] (a1\k);
  }
}
\foreach \i in {1,2,3} {
  \draw[->, green!60!black, thick] (a\i1) -- (a\i2);
}
\node at (0,-2.3) {Layer 1};
\node at (3.5,-2.3) {Layer 2};
\node at (7,-2.3) {Layer 3};
\end{tikzpicture}
\caption{The tournament $T_{6,2}$: a 2-connected tournament on 6 vertices. Blue: forward arcs; red: backward skip-arcs; green: transitive tournament within each layer.}
\end{figure}

To see that the bound in Theorem 3.1 is sharp apart from an additive constant, we prove sharpness for the tournament $T_{n,k}$ using the proposition described in [6] and compute its average eccentricity.
\newpage
\begin{proposition}
Let $k \ge 1$ and let $n$ be a multiple of $k$ with $n \ge 3k$; write $m = n/k$ (so $m \ge 3$). Let
\[
V(T_{n,k}) = \{a_{i,j} : 1 \le i \le m, 1 \le j \le k\},
\]
with arcs $a_{i,j} \to a_{i',j'}$ if and only if one of the following holds:
\begin{itemize}
\item[(i)] $i' = i+1$ (forward to the next layer);
\item[(ii)] $i' \le i-2$ (backward skip-arcs);
\item[(iii)] $i' = i$ and $j < j'$ (transitive tournament within each layer).
\end{itemize}
Then $T_{n,k}$ is a tournament, it is $k$-connected, and
\begin{equation}
\avec(T_{n,k}) = \frac{n}{2k} + O(1).
\end{equation}
\end{proposition}

\begin{proof}
For distinct vertices $a_{i,j}, a_{i',j'}$: if $i' = i$, exactly one of $j < j'$, $j' < j$ holds, giving exactly one of the arcs in (iii); if $i' = i+1$, (i) gives $a_{i,j} \to a_{i+1,j'}$ and none of (i)--(iii) gives the reverse; if $i' = i-1$, (i) applied to the pair $(a_{i-1,j'}, a_{i,j})$ gives $a_{i-1,j'} \to a_{i,j}$, since $i = (i-1)+1$; the reverse $a_{i,j} \to a_{i-1,j'}$ would require (i) ($i-1 = i+1$, false), (ii) ($i-1 \le i-2$, false), or (iii) (different layers, false). So exactly one direction holds. If $|i-i'| \ge 2$, say $i' \le i-2$, then (ii) gives $a_{i,j} \to a_{i',j'}$; the reverse would need $i \le i'-2 \le i-4$ (false), $i = i'+1$ (false since $i-i' \ge 2$), or same layer (false). Thus exactly one arc joins each pair: $T_{n,k}$ is a tournament.

Each layer has $k$ vertices. Deleting fewer than $k$ vertices leaves every layer nonempty. The surviving digraph is strongly connected: if the target layer is ahead, follow forward arcs through surviving vertices; if it is at least two layers behind, use a direct skip arc; if it is the previous layer, go forward one layer and skip back (or, from the last layer, skip back two layers and go forward, using $m \ge 3$); and within a layer, the same forward-and-skip detours give a path of length at most $3$. Hence $T_{n,k}$ is $k$-connected. (This is the argument of [6].)

We compute eccentricities. Let $x = a_{i,j}$ (1-indexed layers).
\begin{itemize}
\item To reach layer $m$: the only forward arcs go to the next layer, so $d(x, a_{m,j'}) \ge m-i$ for $i < m$, and the path $a_{i,j} \to a_{i+1,1} \to \cdots \to a_{m,j'}$ achieves $m-i$. Hence $\max_{j'} d(x, a_{m,j'}) = m-i$ for $i < m$, and at most $3$ for $i = m$ (see below).
\item To reach layer $\ell < i$: if $\ell \le i-2$, $d = 1$ by (ii) (consistent with Lemma 3.2); if $\ell = i-1$, then $x \to a_{i+1,1} \to a_{i-1,j'}$ has length $2$ when $i < m$, and $x \to a_{m-2,1} \to a_{m-1,j'}$ has length $2$ when $i = m$ (using $m \ge 3$).
\item Within layer $i$: $d \le 3$. Indeed, if $j' \ge j$ this is direct (or $0$ if $j' = j$); if $j' < j$, go $a_{i,j} \to a_{i+1,1} \to a_{i+2,1} \to a_{i,j'}$ when $i \le m-2$, $a_{m-1,j} \to a_{m,1} \to a_{m-2,1} \to a_{m-1,j'}$ when $i = m-1$, and $a_{m,j} \to a_{m-2,1} \to a_{m-1,1} \to a_{m,j'}$ when $i = m$.
\end{itemize}
Consequently
\[
\operatorname{ecc}(a_{i,j}) = (m-i) + O(1), \qquad 1 \le i \le m,
\]
the maximum being attained (up to $O(1)$) at the last layer. Averaging over the $m$ layers ($k$ vertices each),
\[
\avec(T_{n,k}) = \frac{1}{m} \sum_{i=1}^m (m-i) + O(1) = \frac{1}{m} \cdot \frac{m(m-1)}{2} + O(1) = \frac{m-1}{2} + O(1) = \frac{n}{2k} + O(1),
\]
proving (8). Thus the constant $1/(2k)$ in Theorem 3.1 cannot be improved.
\end{proof}

\begin{remark}
The computation above is the exact analogue for average eccentricity of the sharpness computation in [6] for mean distance, where $\mu(T_{n,k}) = \frac{n}{6k} + O(1)$. The tournament $T_{n,k}$ is simultaneously extremal for both parameters. Here, $T_{n,k}$ has minimum out-degree exactly $k$ ($\delta^+ = k$), so the $\frac{1}{2k}$ coefficient is best possible even for tournaments of minimum out-degree $\ge k$.
\end{remark}

\begin{figure}[htbp]
\centering
\begin{tikzpicture}[scale=0.85, v/.style={circle, draw, minimum size=0.8cm, inner sep=1pt, font=\small}]
\foreach \i/\x in {1/0, 2/3.5, 3/7} {
  \foreach \j/\y in {1/1.7, 2/0, 3/-1.7} {
    \node[v] (a\i\j) at (\x,\y) {$a_{\i,\j}$};
  }
}
\foreach \j in {1,2,3} {
  \foreach \k in {1,2,3} {
    \draw[->, blue, thick] (a1\j) -- (a2\k);
    \draw[->, blue, thick] (a2\j) -- (a3\k);
    \draw[->, red, thick] (a3\j) to[bend left=10] (a1\k);
  }
}
\foreach \i in {1,2,3} {
  \draw[->, green!60!black, thick] (a\i1) -- (a\i2);
  \draw[->, green!60!black, thick] (a\i2) -- (a\i3);
  \draw[->, green!60!black, thick] (a\i1) to[bend left=30] (a\i3);
}
\node at (0,-2.8) {Layer 1};
\node at (3.5,-2.8) {Layer 2};
\node at (7,-2.8) {Layer 3};
\end{tikzpicture}
\caption{The tournament $T_{9,3}$: a 3-connected tournament on 9 vertices (three layers of three). Blue: complete bipartite forward arcs; red: backward skip-arcs from layer 3 to layer 1; green: transitive tournament $a_{i,1} \to a_{i,2} \to a_{i,3}$ within each layer.}
\end{figure}
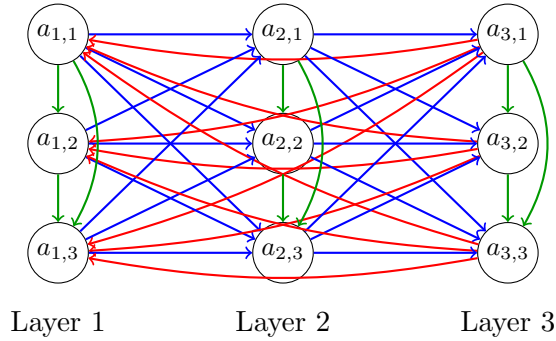

\subsection{Bounds on strong tournaments}

\begin{theorem}
Let $T$ be a strong tournament of order $n \ge 3$. Then
\[
\avec(T) \ge 2,
\]
and the constant $2$ is best possible.
\end{theorem}

\begin{proof}
Let $v \in V(T)$. If $\operatorname{ecc}_T(v) = 1$, then $v$ dominates every other vertex, so $v$ has in-degree $0$ and there is no directed path from any other vertex to $v$, contradicting the strong connectivity of $T$. Hence $\operatorname{ecc}_T(v) \ge 2$ for every $v \in V(T)$, and averaging over $V(T)$ gives
\[
\avec(T) = \frac{1}{n} \sum_{v \in V(T)} \operatorname{ecc}_T(v) \ge 2.
\]
In particular, this holds for $k$-connected tournaments, since every $k$-connected tournament is strong.

We now show that the constant $2$ cannot be improved. Let $n = 2m+1$ with $m \ge 1$, and let $T$ be the tournament with vertex set $\mathbb{Z}_n$ in which $i \to j$ if and only if $j-i \pmod n \in \{1,\dots,m\}$. This is indeed a tournament: for distinct $i,j$, the residues $j-i$ and $i-j$ sum to $n = 2m+1$, so exactly one of them lies in $\{1,\dots,m\}$. We claim that $\operatorname{ecc}_T(i) = 2$ for every $i$. Let $j \ne i$ and let $d \equiv j-i \pmod n$ with $1 \le d \le 2m$. If $1 \le d \le m$, then $i \to j$ directly. If $m+1 \le d \le 2m$, write $d = m+t$ with $1 \le t \le m$; then $i \to i+t$ and $(i+t) \to j$, since $j-(i+t) \equiv m \pmod n$, so $d_T(i,j) = 2$. In particular, for $j$ with $d = m+1$, we have $d_T(i,j) = 2$, so $\operatorname{ecc}_T(i) \ge 2$; combined with $d_T(i,j) \le 2$ for all $j$, we get $\operatorname{ecc}_T(i) = 2$ for every $i$. Consequently
\[
\avec(T) = \frac{1}{n} \sum_{i \in V(T)} 2 = 2.
\]
Hence no constant larger than $2$ can serve as a universal lower bound for the average eccentricity of strong tournaments, so the bound $2$ is best possible. Moreover, taking $m = k$, this tournament has minimum in- and out-degree $k$ and is $k$-connected (for tournaments, the vertex-connectivity equals the minimum in-/out-degree), so the bound $2$ is best possible for $k$-connected tournaments of each fixed connectivity as well.
\end{proof}

\begin{theorem}
Let $T$ be a strong tournament of order $n \ge 3$. Then
\begin{equation}
\avec(T) \le \frac{n+11}{2} + \frac{1}{8n}.
\end{equation}
The upper bound is best possible apart from the additive constant.
\end{theorem}

\begin{remark}
A tournament is strong (strongly connected) if and only if it is 1-connected: 1-connectivity means the digraph remains strongly connected after deleting any set of fewer than 1 vertices, i.e.\ after deleting nothing, which is precisely strong connectivity. Thus Theorem 3.3 is the case $k = 1$ of Theorem 3.1.
\end{remark}

\begin{proof}[Proof of Theorem 3.3]
We prove the upper bounds using already obtained propositions and lemmas.

Since strong = 1-connected, Theorem 3.1 with $k = 1$ gives $\avec(T) < \frac{n}{2} + \frac{11}{2} + \frac{1}{8n}$, and a fortiori the non-strict inequality
\[
\avec(T) \le \frac{n}{2} + \frac{11}{2} + \frac{1}{8n}.
\]
For the reader's convenience we sketch the $k = 1$ specialization directly. Let $D = \diam(T)$ and $V_i, n_i$ as in Section 3. Strong connectivity gives $n_i \ge 1$ for $1 \le i \le D-1$ (every interior layer is nonempty). By Proposition 3.2, $\operatorname{ecc}_T(x) \le (D-i) + 5$ for $x \in V_i$, so
\[
n \cdot \avec(T) \le \sum_{i=0}^D n_i(D-i) + 5n.
\]
Now $\sum_{i=0}^D n_i(D-i) = Dn - \sum_{i=0}^D i n_i$ and, $n_i \ge 1$ being automatic, applying Lemma 3.3 with $k = 1$ gives:
\[
\sum_{i=0}^D n_i(D-i) \le \frac{n^2}{2} + \frac{n}{2} + \frac{1}{8}.
\]
Indeed, with $k = 1$ the concave function is $g(D) = -\frac{1}{2}D^2 + \left(n + \frac{1}{2}\right)D$, maximized at $D^* = n + \frac{1}{2}$ with value $\frac{n^2}{2} + \frac{n}{2} + \frac{1}{8}$. Hence
\[
n \cdot \avec(T) \le \frac{n^2}{2} + \frac{n}{2} + \frac{1}{8} + 5n,
\]
i.e.\
\[
\avec(T) \le \frac{n}{2} + \frac{11}{2} + \frac{1}{8n},
\]
as claimed. (The strict inequality of Theorem 3.1 also holds here; we state the non-strict form to match Plesn\'ik's $\le$.)

Sharpness. Proposition 3.3 with $k = 1$ yields the strong tournament $T_{n,1}$: $n$ layers of size $1$, i.e.\ the ``path tournament'' $a_1 \to a_2 \to \cdots \to a_n$ with all backward arcs $a_i \to a_j$ for $j \le i-2$. A direct computation gives
\[
\avec(T_{n,1}) = \frac{1}{n} \sum_{i=1}^n (n-i) + O(1) = \frac{n-1}{2} + O(1) = \frac{n}{2} + O(1),
\]
so the leading constant $\frac{1}{2}$ cannot be improved. Up to isomorphism, $T_{n,1}$ is the unique strong tournament of diameter $n-1$, exactly as in Plesn\'ik's theorem.

Note that Theorem 3.3 is the average-eccentricity analogue of Plesn\'ik's theorem [10] (Theorem 1 in [6]), which states $\frac{3}{2} \le \mu(T) \le \frac{n+4}{6} + \frac{1}{n}$ for strong tournaments. Plesn\'ik further showed $\mu(T) = \frac{3}{2}$ iff $\diam(T) = 2$, and the upper bound is attained iff $T$ is the unique strong tournament of diameter $n-1$ (which is $T_{n,1}$). The same extremal $T_{n,1}$ shows our $\frac{n}{2}$ upper bound for $\avec$ to be best possible.
\end{proof}

\begin{figure}[htbp]
\centering
\begin{tikzpicture}[scale=0.9, v/.style={circle, draw, minimum size=0.8cm, inner sep=1pt, font=\small}]
\foreach \i/\x in {1/0, 2/2.2, 3/4.4, 4/6.6, 5/8.8} {
  \node[v] (a\i) at (\x,0) {$a_\i$};
  \node at (\x,1.1) {Layer \i};
}
\foreach \i/\j in {1/2, 2/3, 3/4, 4/5} {
  \draw[->, blue, thick] (a\i) -- (a\j);
}
\draw[->, red, thick] (a3) to[bend left=25] (a1);
\draw[->, red, thick] (a4) to[bend left=25] (a2);
\draw[->, red, thick] (a5) to[bend left=25] (a3);
\draw[->, red, thick] (a4) to[bend left=45] (a1);
\draw[->, red, thick] (a5) to[bend left=45] (a2);
\draw[->, red, thick] (a5) to[bend left=60] (a1);
\end{tikzpicture}
\caption{The tournament $T_{5,1}$: a strong tournament ($k = 1$) on 5 vertices. Blue arrows go forward to the next layer; red arrows are backward skip-arcs to layers at distance at least two.}
\end{figure}
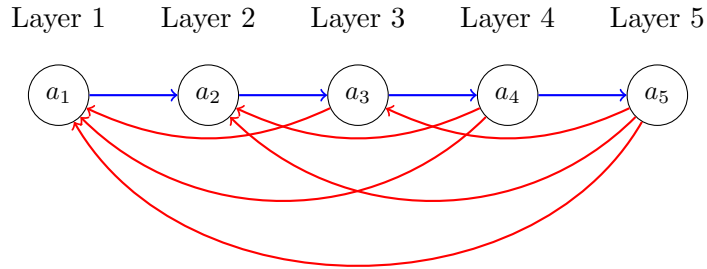

\section{Concluding remarks}

Proposition 3.3 shows that $T_{n,k}$ attains the leading term $n/(2k)$ in Theorem 3.1. In fact, the eccentricities computed in the proof of Proposition 3.3 give the exact value of $\avec(T_{n,k})$. Indeed, for $m = n/k \ge 3$, a vertex $a_{i,j}$ in layer $i$ has eccentricity $\max(m-i,2)$ if $j = 1$ and $\max(m-i,3)$ if $j > 1$; summing over all $i$ and $j$ gives
\[
\avec(T_{n,k}) = \frac{n}{2k} - \frac{1}{2} + \frac{6k-3}{n}.
\]
We conjecture that this value is the largest possible.

\begin{conjecture}
Let $k \ge 1$ and let $n$ be a multiple of $k$ with $n \ge 3k$. Then every $k$-connected tournament $T$ of order $n$ satisfies
\[
\avec(T) \le \avec(T_{n,k}) = \frac{n}{2k} - \frac{1}{2} + \frac{6k-3}{n}.
\]
\end{conjecture}

\end{document}